\documentclass[11pt]{article}
\usepackage{amsmath,amssymb,amsthm,amsfonts}
\usepackage[margin=1in]{geometry}
\usepackage{enumitem}

\newtheorem{theorem}{Theorem}
\newtheorem{lemma}{Lemma}
\newtheorem{proposition}{Proposition}

\newtheorem{remark}{Remark}
\newtheorem{problem}{Problem}
\usepackage{lmodern}

\usepackage{hyperref}
\hypersetup{
	colorlinks=true,
	linkcolor=blue,
	filecolor=magenta,
	urlcolor=cyan,
	citecolor=red
} 
\title{A sharp fixed-size spectral bound for $kK_3$-free graphs}
\author{Joyentanuj Das\thanks{Corresponding author. Emails:
joyentanuj@gmail.com and joyentad@srmist.edu.in.}
\quad and \quad
Yamini V\thanks{Email: yv1053@srmist.edu.in.}\\[2pt]
\small Department of Mathematics, College of Engineering and Technology,\\[-1pt]
\small SRM Institute of Science and Technology, Kattankulathur,
Chennai 603203, India}
\date{}

\begin{document}
\maketitle
\begin{abstract}
For a fixed integer $k\ge2$, we establish a sharp adjacency-spectral upper
bound for sufficiently large $m$-edge $kK_3$-free graphs. We prove
\[
\lambda(G)\le (k-1)+\sqrt{m-k(k-1)}.
\]
Moreover, equality holds precisely when $(2k-1)\mid m$ and, up to isolated vertices,
$G$ is the join of $K_{2k-1}$ with an independent set of
$m/(2k-1)-(k-1)$ vertices. The case $k=2$ was previously known; our
argument establishes every fixed $k\ge3$. The proof requires information
beyond first-order spectral stability. We derive an exact nonnegative
defect identity at a maximum-Perron vertex, use it to bound the entire
outer layer by a constant, and reduce the remaining graph to a bounded
core with finitely many independent twin classes. A Perron-vector
concentration identity and the Erd\H{o}s--Gallai matching theorem then
force the unique extremal core. A nearly extremal family lies only
$\Theta(m^{-1/2})$ below the target, showing why an exact second-order
analysis is necessary.
\end{abstract}



\section{Introduction}

All graphs considered in this paper are finite, undirected, and simple. Let $G=(V(G),E(G))$ be a graph with vertex set $V(G)$ and edge set $E(G)$, and write $n=|V(G)|$ for its order and $m=e(G)=|E(G)|$ for its size. For a vertex $v\in V(G)$, $N_G(v)$ denotes its neighborhood and $d_G(v)=|N_G(v)|$ its degree; when the ambient graph is clear we drop the subscript. For a vertex subset $X\subseteq V(G)$, $G[X]$ denotes the subgraph of $G$ induced by $X$, and we write $e(X):=e(G[X])$ for the number of edges inside $X$. For two disjoint vertex subsets $X,Y\subseteq V(G)$, $e(X,Y)$ denotes the number of edges with one endpoint in $X$ and the other in $Y$, and for $v\in V(G)$ we write $d_X(v):=|N_G(v)\cap X|$.

For two graphs $G$ and $H$, $G\cup H$ denotes their vertex-disjoint union, and $G\vee H$ denotes their \emph{join}, obtained from the disjoint union $G\cup H$ by adding all edges between $V(G)$ and $V(H)$. For a graph $F$ and an integer $k\ge1$, $kF$ denotes the disjoint union of $k$ vertex-disjoint copies of $F$; in particular $kK_3$ denotes $k$ vertex-disjoint triangles. We write $K_n$ for the complete graph on $n$ vertices and $K_n^c$ (or $nK_1$) for the empty graph (independent set) on $n$ vertices.

The \emph{cone} over a graph $H$ is the join $K_1\vee H$; the vertex
belonging to the $K_1$ is called the \emph{apex} of the cone. Two
distinct nonadjacent vertices are called \emph{independent twins} if
they have the same neighborhood. An \emph{independent twin class} is
an independent set whose vertices all have the same neighborhood.

A \emph{triangle packing} in a graph $G$ is a collection
$T_1,\ldots,T_s$ of pairwise vertex-disjoint triangles in $G$; its
\emph{size} is the number $s$ of triangles in the collection. Thus
$G$ is $kK_3$-free exactly when it contains no triangle packing of
size $k$, or equivalently, when every triangle packing in $G$ has size
at most $k-1$.

A \emph{matching} is a collection of pairwise vertex-disjoint edges.
The maximum size of a matching in $G$ is its \emph{matching number},
denoted by $\nu(G)$. A matching is \emph{perfect} if every vertex of
$G$ is incident with one of its edges. A \emph{vertex cover} is a set
of vertices meeting every edge of the graph.

A graph $G$ is said to be \emph{$F$-free} if $G$ does not contain $F$ as a (not necessarily induced) subgraph. Given a family of graphs $\mathcal{F}$, $G$ is $\mathcal F$-free if it is $F$-free for every $F\in\mathcal F$.

The \emph{adjacency matrix} $A(G)$ of $G$ is the $n\times n$ $0$-$1$ matrix indexed by $V(G)$ with $A(G)_{uv}=1$ if $uv\in E(G)$ and $0$ otherwise. Since $A(G)$ is real and symmetric, its eigenvalues are real; we denote the largest of them by $\lambda(G)$, called the \emph{spectral radius} of $G$. When $G$ is connected, $A(G)$ is irreducible and nonnegative, so by the Perron--Frobenius theorem $\lambda(G)$ is a simple eigenvalue admitting a positive eigenvector $x=(x_v)_{v\in V(G)}$, unique up to positive scaling, called the \emph{Perron vector} of $G$. The Perron vector satisfies the eigenequation
\[
\lambda(G)\,x_v \;=\; \sum_{u\in N_G(v)}x_u \qquad\text{for every } v\in V(G).
\]

For a nonzero vector $z\in\mathbb R^{V(G)}$, its \emph{Rayleigh
quotient} with respect to $G$ is
\[
\mathcal R_G(z):=\frac{z^\top A(G)z}{z^\top z}.
\]
The variational characterization of the largest eigenvalue states that
\[
\lambda(G)=\max_{z\ne0}\mathcal R_G(z).
\]

A partition $V(G)=V_1\cup\cdots\cup V_t$ into nonempty cells is
\emph{equitable} if, for every $i,j$, each vertex of $V_i$ has the
same number $q_{ij}$ of neighbors in $V_j$. The matrix
$Q=(q_{ij})_{1\le i,j\le t}$ is called the \emph{quotient matrix}.
An eigenvector of $Q$ lifts to an eigenvector of $A(G)$ that is
constant on each cell $V_i$.

Now, we recall the classical and spectral Tur\'an frameworks and summarize the fixed-order results most relevant to forbidding disjoint triangles.

\subsection{The spectral Tur\'an problem and its fixed-size analogue}

A central topic in extremal graph theory is the \emph{Tur\'an problem}: given a forbidden graph (or family) $F$, determine
\[
\mathrm{ex}(n,F)\;=\;\max\{e(G): |V(G)|=n,\ G \text{ is } F\text{-free}\},
\]
the maximum number of edges in an $n$-vertex $F$-free graph. The \emph{spectral Tur\'an problem}, initiated by Nikiforov and others, replaces the edge count by the spectral radius:
\[
\mathrm{spex}(n,F)\;=\;\max\{\lambda(G): |V(G)|=n,\ G\text{ is } F\text{-free}\}.
\]
Nikiforov \cite{Nikiforov2007} established the spectral Tur\'an theorem for $K_{r+1}$-free graphs, extending the classical Tur\'an theorem to the spectral setting, and later \cite{Nikiforov2010} studied the spectral radius of graphs forbidding paths or cycles of prescribed length. Since then a large body of work has developed spectral analogues of classical Tur\'an-type results; see the survey of Li, Liu and Feng \cite{LiLiuFeng} for a broad overview.

For \emph{disconnected} forbidden graphs of matching type---disjoint unions of several copies of a fixed graph, or more generally disjoint unions of simple components---a parallel line of research has emerged. In the fixed-order (spectral Tur\'an) setting, Feng, Yu and Zhang \cite{FengYuZhang} determined the graphs of maximum spectral radius among all graphs with a given matching number. Chen, Liu and Zhang obtained related results for linear forests \cite{ChenLiuZhang1} and later for star forests \cite{ChenLiuZhang2}. For disjoint unions of cliques, Ni, Wang and Kang \cite{NiWangKang} proved that, for sufficiently large $n$, the unique $n$-vertex graph of maximum spectral radius among all $kK_{r+1}$-free graphs is $K_{k-1}\vee T_{n-k+1,r}$, where $T_{n-k+1,r}$ denotes the Tur\'an graph on $n-k+1$ vertices with $r$ parts. Lei and Li \cite{LeiLi} extended this to disjoint unions of color-critical graphs, and Wang, Ni, Kang and Fan \cite{WangNiKangFan} obtained further spectral extremal results for edge blow-ups of star forests.

Specializing the theorem of Ni, Wang and Kang \cite{NiWangKang} to $r=2$ (so that $K_{r+1}=K_3$) shows that, for sufficiently large $n$, the unique $n$-vertex $kK_3$-free graph of maximum spectral radius is $K_{k-1}\vee T_{n-k+1,2}$, a clique of size $k-1$ joined to a balanced complete bipartite-type graph. This is the fixed-\emph{order} extremal structure for $kK_3$; the present paper instead studies the fixed-\emph{size} problem for the same forbidden family.

We next formulate the fixed-size spectral problem and review the results that lead to the present $kK_3$-free setting.

\subsection{The Brualdi--Hoffman--Tur\'an problem}

If one fixes the number of \emph{edges} $m$ rather than the number of vertices $n$, the corresponding extremal problem becomes
\[
\max\{\lambda(G): e(G)=m,\ G\text{ is } F\text{-free}\},
\]
often called the \emph{Brualdi--Hoffman--Tur\'an problem}, after Brualdi and Hoffman's original study of the analogous question without a forbidden subgraph constraint. Since isolated vertices affect neither $e(G)$ nor $\lambda(G)$, it is natural to focus on the nontrivial part of an extremal graph, i.e.\ to determine the extremal structure up to the addition of isolated vertices.

Compared with the fixed-order spectral Tur\'an problem, the fixed-size problem is typically more delicate: one no longer works within a prescribed vertex set, so edges may be redistributed among an unbounded number of vertices, and the extremal graphs in the two settings frequently have genuinely different structures. This divergence is a recurring theme in the literature and is also the central phenomenon exhibited by the results of this paper.

Exact results for the Brualdi--Hoffman--Tur\'an problem have been obtained for several forbidden (families of) graphs. Zhai, Lin and Shu \cite{ZhaiLinShu} established sharp spectral conditions forcing certain complete bipartite graphs, short cycles, and consecutive cycle lengths, and posed conjectures for longer forbidden cycles. Li, Zhai and Shu \cite{LiZhaiShu} subsequently resolved those cycle conjectures (and stronger variants for cycles with a chord) for sufficiently large size. Li and Yu \cite{LiYu} treated graphs simultaneously forbidding a triangle and a pentagon, while Lu, Lu and Li \cite{LuLuLi} determined the extremal graphs forbidding $C_7$ or the ``theta-type'' graph $C_6^\triangle$. More recently, Li, Zhao and Zou \cite{LiZhaoZou} considered graphs forbidding a fan, a friendship graph, or a theta graph. For disconnected acyclic forbidden graphs, Zhang and Wang \cite{ZhangWang} obtained exact results for star forests.

Prior to the present work, for disconnected non-bipartite forbidden graphs of matching type the fixed-size problem had been resolved only for $2K_3$, the disjoint union of two triangles, which is the simplest such graph. In~\cite{Wang2026}, Wang, Jia and Ni showed that if $m\ge18$ and $G$ is a $2K_3$-free graph with $m$ edges, then
\[
\lambda(G)\ \le\ 1+\sqrt{m-2},
\]
with equality if and only if $m\equiv0\pmod3$ and $G\cong\big(K_3\vee\tfrac{m-3}{3}K_1\big)\cup tK_1$ for some $t\ge0$. Notably, the extremal graph $K_3\vee\tfrac{m-3}{3}K_1$ (a triangle joined to an independent set) is \emph{not} the fixed-size restriction of the fixed-order extremal graph $K_1\vee T_{n-1,2}$ described above; this already illustrates, in the smallest nontrivial case $k=2$, that the fixed-order and fixed-size extremal problems for $kK_3$-free graphs have different answers.

\subsection{Main result}

We now state our sharp fixed-size bound for the full family $kK_3$ and
describe all graphs attaining equality. The case
$k=2$ was established by Wang, Jia and Ni~\cite{Wang2026}; the content of
the present paper is the inductive step for every fixed $k\ge3$. Besides
proving the sharp bound, the theorem identifies every
equality case up to isolated vertices.

\begin{theorem}\label{thm:main}
For every integer $k\ge2$ there is a constant $M_0=M_0(k)$ such that the following holds. Let $m\ge M_0$ and let $G$ be a $kK_3$-free graph with $e(G)=m$. Then
\[
\lambda(G)\le (k-1)+\sqrt{m-k(k-1)}.
\]
Equality holds if and only if $(2k-1)\mid m$ and
\[
G\cong\Big(K_{2k-1}\vee qK_1\Big)\cup tK_1,
\qquad q=\frac{m}{2k-1}-(k-1),
\]
for some $t\ge0$. In particular, the nontrivial component of every equality graph is unique.
\end{theorem}

The extremal graph is structurally different from its fixed-order counterpart: almost all vertices form an independent set joined to a clique of order $2k-1$. Thus Theorem~\ref{thm:main} is not a reformulation of the known fixed-order spectral Tur\'an theorem.

We next use a nearly extremal construction to explain why an exact proof requires information beyond first-order spectral stability.

\subsection{Why first-order stability is insufficient}
Recent general edge-spectral stability theorems show that an $F$-free
graph at the $\chi(F)=3$ threshold is $o(m)$ edges away from a complete
bipartite graph \cite{LiLiuZhangStability}. This first-order conclusion
does not see the $O(\sqrt m)$ term which determines the present problem.
The later exact theory for almost-bipartite forbidden graphs
\cite{LiLiuZhangCritical} does not apply to $kK_3$, which cannot be made
bipartite by deleting one edge. Likewise, the bounded-core theorem for
finite degenerate forbidden families \cite{FangLinZhai} requires the
family to contain a bipartite member. The bounded-outer-layer and
twin-class argument below is a problem-specific second-order replacement:
the defect identity first supplies an $O_k(1)$ core, and the Perron
equations then identify the core vertices on which the normalized
Perron mass asymptotically concentrates.

The need for a genuinely second-order argument can be seen from a nearly
extremal family. Put $j=k-1$ and let
\[
H_c:=K_j\vee K_{c,c},
\]
where the two parts of $K_{c,c}$ are denoted by $X$ and $Y$. Since
$K_{c,c}$ is bipartite, every triangle of $H_c$ contains at least one
vertex of the clique $K_j$. As $j=k-1$, any collection of $k$ pairwise
vertex-disjoint triangles would require at least $k$ distinct vertices of
$K_j$, which is impossible. Hence $H_c$ is $kK_3$-free.

The edge count consists of the $c^2$ edges between $X$ and $Y$, the
$2jc$ edges joining $K_j$ to $X\cup Y$, and the edges inside $K_j$.
Therefore
\[
e(H_c)=c^2+2jc+\binom j2.
\]
The partition $V(H_c)=V(K_j)\cup X\cup Y$ is equitable, with quotient
matrix
\[
Q=\begin{pmatrix}
j-1&c&c\\
j&0&c\\
j&c&0
\end{pmatrix}.
\]
By symmetry, the Perron vector has the same coordinate on $X$ and $Y$.
Thus the Perron eigenvalue is the largest eigenvalue of the reduced matrix
\[
\widetilde Q=\begin{pmatrix}j-1&2c\\ j&c\end{pmatrix}.
\]
Consequently, $\lambda=\lambda(H_c)$ satisfies
\[
(\lambda-j+1)(\lambda-c)-2jc=0,
\]
and hence
\begin{equation}
\lambda(H_c)
=\frac{c+j-1+\sqrt{c^2+(6j+2)c+(j-1)^2}}{2}.                       \label{eq:near-extremal-root}
\end{equation}

For fixed constants $\alpha$ and $\beta$, the Taylor expansion
\[
\sqrt{c^2+\alpha c+\beta}
=c+\frac{\alpha}{2}
+\left(\frac{\beta}{2}-\frac{\alpha^2}{8}\right)\frac1c
+O(c^{-2})
\]
holds as $c\to\infty$. Applying it to
\eqref{eq:near-extremal-root}, with $\alpha=6j+2$ and
$\beta=(j-1)^2$, gives
\[
\sqrt{c^2+(6j+2)c+(j-1)^2}
=c+3j+1-\frac{4j(j+1)}c+O(c^{-2}).
\]
It follows that
\begin{equation}
\lambda(H_c)
=c+2j-\frac{2j(j+1)}c+O(c^{-2})
=c+2j-\frac{2k(k-1)}c+O(c^{-2}).                                  \label{eq:near-extremal-lambda}
\end{equation}

We next expand the target value. Since $k=j+1$, we have
\begin{align*}
e(H_c)-k(k-1)
&=c^2+2jc+\frac{j(j-1)}2-j(j+1)\\
&=c^2+2jc-\frac{j(j+3)}2.
\end{align*}
Using the same Taylor expansion, now with $\alpha=2j$ and
$\beta=-j(j+3)/2$, yields
\[
\sqrt{e(H_c)-k(k-1)}
=c+j-\frac{3j(j+1)}{4c}+O(c^{-2}).
\]
Therefore
\begin{equation}
\lambda_0:=j+\sqrt{e(H_c)-k(k-1)}
=c+2j-\frac{3k(k-1)}{4c}+O(c^{-2}).                               \label{eq:near-extremal-target}
\end{equation}
Subtracting \eqref{eq:near-extremal-lambda} from
\eqref{eq:near-extremal-target} gives
\[
\lambda_0-\lambda(H_c)
=\frac{5k(k-1)}{4c}+O(c^{-2}).
\]
Finally, $e(H_c)=c^2+2jc+O(1)$ implies
$\sqrt{e(H_c)}=c+j+O(c^{-1})$, and hence
\[
\frac1c=\frac1{\sqrt{e(H_c)}}+O(e(H_c)^{-1}).
\]
Thus
\[
\lambda_0-\lambda(H_c)
=\frac{5k(k-1)}{4\sqrt{e(H_c)}}+O(e(H_c)^{-1}).
\]
The deficit is only of order $m^{-1/2}$: the two spectral radii agree in
their leading and constant-order terms and separate only at the next
order. This explains why first-order stability alone cannot settle the
exact problem and why the cone threshold introduced below must be taken
large.

We finish the introduction by outlining the four reductions in the proof and indicating which features may be useful in other fixed-size problems.

\subsection{Proof strategy and methodological contribution}

The proof is inductive, with the published $2K_3$ theorem as its base.
Its main difficulty is that first-order stability leaves an error of order
$o(m)$, whereas the example above shows that the relevant spectral gap is
only of order $m^{-1/2}$. We overcome this through four successive
reductions.

First, for a spectral maximizer and a vertex $u^*$ of maximum Perron
coordinate, put
\[
A=N_G(u^*),\qquad B=V(G)\setminus N_G[u^*],\qquad r=k-1.
\]
A bounded triangle packing gives a bounded vertex cover of $G[A]$. This
allows us to isolate the degree core
\[
C=\{v\in A:d_A(v)>2r\}
\]
and to prove $2r\le|C|\le4r$. More importantly, an exact rewriting of
the double eigenequation produces a defect identity: if
$|C|=2r+q$, then six nonnegative structural and Perron-vector defects
have total at most $\binom q2$. When $q=0$, every defect vanishes and
the target graph follows immediately.

Second, suppose $q>0$. The defect identity directly bounds $e(B)$ but
does not bound $|B|$. A missing-Perron-mass argument, combined with a
triangle construction using hubs and a matching, converts this weighted
information into the
uniform estimate $|B|=O_k(1)$. This bounded-outer-layer step is the key
compactness mechanism in the proof.

Third, after placing $B$, $C$, and the endpoints of the boundedly many
edges in the low-degree part into a bounded set $K$, every remaining
vertex belongs to an independent twin class determined by its
neighborhood in $K$. Eliminating these twin-class coordinates from the
Perron equations gives an exact finite-dimensional matrix identity. It
forces the normalized Perron vector on $K$ to concentrate on a
neighborhood type occurring linearly many times.

Finally, the Erd\H{o}s--Gallai matching theorem shows that this dominant
type must induce $K_{2k-1}$. If $q>0$, the resulting perfect matching,
together with one extra core vertex, creates $k$ disjoint triangles, a
contradiction. Hence $q=0$.

Although the numerical thresholds and the final packing argument are
specific to triangles, the proof separates three mechanisms that may be
useful in other fixed-size problems with disconnected forbidden graphs:
an exact local defect identity, conversion of weighted defect into a
bounded outer layer, and a finite-core closure through twin classes. We
do not claim such an extension here, but this separation is one reason
for presenting the intermediate lemmas in a modular form.

Section~2 collects the preliminary tools. Section~3 proves the four
reductions above, and Section~4 completes the induction and the equality
characterization.

\section{Preliminaries}

We collect here the tools used throughout the proof: a spectral-radius-increasing edge shift, the Erd\H{o}s--Gallai matching bound in the form needed later, the spectral radius of the target extremal graph, and a lemma quantifying how a clique with a bounded number of fully adjacent hubs forces disjoint triangles once it is joined to a sufficiently large set.

\begin{lemma}[Edge shifting; Wu--Xiao--Hong \cite{WuXiaoHong}]\label{lem:shift}
Let $G$ be connected with Perron vector $x$, and let $u,v,w$ be distinct
vertices. If $uv\in E(G)$, $x_w\ge x_v$, and $uw\notin E(G)$, then
deleting $uv$ and adding $uw$ strictly increases the spectral radius.
\end{lemma}

\begin{proof}
Let $G'=G-uv+uw$ and scale $x$ to be a unit vector. Then
\[
x^\top A(G')x-x^\top A(G)x=2x_u(x_w-x_v)\ge0,
\]
so $\lambda(G')\ge\lambda(G)$. If equality held, $x$ would attain the
largest Rayleigh quotient of $A(G')$ and hence would be an eigenvector of
$G'$ for $\lambda(G)$. At the vertex $v$, however,
\[
(A(G')x)_v=(A(G)x)_v-x_u=\lambda(G)x_v-x_u<\lambda(G)x_v,
\]
a contradiction. Therefore $\lambda(G')>\lambda(G)$.
\end{proof}

\begin{proposition}[Matching extremum]\label{lem:matching}
If a graph $H$ on $2\ell$ vertices has no perfect matching, then
\[
e(H)\le\binom{2\ell-1}{2}.
\]
Equivalently, at least $2\ell-1$ edges of $K_{2\ell}$ are absent from $H$.
\end{proposition}

\begin{proof}
Because $H$ has $2\ell$ vertices, a perfect matching in $H$ would have
exactly $\ell$ edges. Thus the assumption that $H$ has no perfect matching
is equivalent to
\[
\nu(H)\le \ell-1,
\]
where $\nu(H)$ denotes the matching number of $H$.

The Erd\H{o}s--Gallai matching theorem \cite{ErdosGallai1959} states that
an $n$-vertex graph with matching number at most $a$ has at most
\[
\max\left\{\binom{2a+1}{2},\binom{a}{2}+a(n-a)\right\}
\]
edges. We apply this theorem with
\[
n=2\ell\qquad\text{and}\qquad a=\ell-1.
\]
The first of the two possible bounds becomes
\[
\binom{2a+1}{2}=\binom{2\ell-1}{2}.
\]
The second becomes
\begin{align*}
\binom a2+a(n-a)
&=\binom{\ell-1}{2}+(\ell-1)\bigl(2\ell-(\ell-1)\bigr)\\
&=\frac{(\ell-1)(\ell-2)}2+(\ell-1)(\ell+1)\\
&=\frac{3\ell(\ell-1)}2.
\end{align*}
The difference between the first bound and the second is
\begin{align*}
\binom{2\ell-1}{2}-\frac{3\ell(\ell-1)}2
&=(\ell-1)(2\ell-1)-\frac{3\ell(\ell-1)}2\\
&=\frac{(\ell-1)(\ell-2)}2\ge0.
\end{align*}
Therefore the larger Erd\H{o}s--Gallai bound is
$\binom{2\ell-1}{2}$, and hence
\[
e(H)\le\binom{2\ell-1}{2}.
\]

Finally, $K_{2\ell}$ has $\binom{2\ell}{2}$ edges, so the number of its
edges that are absent from $H$ is at least
\[
\binom{2\ell}{2}-\binom{2\ell-1}{2}=2\ell-1.
\]
This also proves the equivalent formulation. The bound is sharp, as shown
by $K_{2\ell-1}\cup K_1$, which has no perfect matching and has exactly
$\binom{2\ell-1}{2}$ edges.
\end{proof}

\begin{lemma}[Target graph]\label{lem:target}
For $q\ge0$,
\[
\lambda\big(K_{2k-1}\vee qK_1\big)=(k-1)+\sqrt{(k-1)^2+(2k-1)q}.
\]
The graph $K_{2k-1}\vee qK_1$ is $kK_3$-free, has $\binom{2k-1}{2}+(2k-1)q$ edges, and if $m=\binom{2k-1}{2}+(2k-1)q$ (equivalently $q=\frac{m}{2k-1}-(k-1)$) then
\[
\lambda\big(K_{2k-1}\vee qK_1\big)=(k-1)+\sqrt{m-k(k-1)}.
\]
\end{lemma}

\begin{proof}
Write
\[
G_q:=K_{2k-1}\vee qK_1,
\]
and let $C$ and $I$ denote, respectively, its clique and independent
parts. We first compute $\lambda(G_q)$. If $q=0$, then
$G_q=K_{2k-1}$, so
\[
\lambda(G_q)=2k-2
=(k-1)+\sqrt{(k-1)^2},
\]
which is the required formula.

Now suppose $q\ge1$. Every vertex of $C$ has $2k-2$ neighbors in
$C$ and all $q$ vertices of $I$ as neighbors. Every vertex of $I$
has all $2k-1$ vertices of $C$ as neighbors and no neighbor in $I$.
Thus the partition $V(G_q)=C\cup I$ is equitable, with quotient matrix
\[
Q=\begin{pmatrix}
2k-2&q\\
2k-1&0
\end{pmatrix}.
\]
Its characteristic polynomial is
\[
\det(xI-Q)=x^2-2(k-1)x-(2k-1)q.
\]
Therefore its larger eigenvalue is
\[
\theta=(k-1)+\sqrt{(k-1)^2+(2k-1)q}.
\]
For completeness, this quotient eigenvalue is indeed the spectral
radius of $G_q$. Since $q\ge1$, the nonnegative matrix $Q$ is
irreducible and has a positive eigenvector $(\alpha,\beta)^\top$
for $\theta$. Assigning coordinate $\alpha$ to every vertex of $C$
and coordinate $\beta$ to every vertex of $I$ produces a positive
eigenvector of $A(G_q)$ with eigenvalue $\theta$. The graph $G_q$ is
connected, so the Perron--Frobenius theorem gives
\[
\lambda(G_q)=\theta.
\]

We next verify that $G_q$ is $kK_3$-free. Since $I$ is independent,
every triangle contains at least two vertices of $C$. Hence $k$
pairwise vertex-disjoint triangles would require at least $2k$
distinct vertices of $C$. This is impossible because $|C|=2k-1$.

Finally, the edges of $G_q$ consist of the edges inside $C$ and all
edges between $C$ and $I$. Consequently,
\[
e(G_q)=\binom{2k-1}{2}+(2k-1)q
=(2k-1)(q+k-1).
\]
If this number is $m$, then
\[
q=\frac{m}{2k-1}-(k-1).
\]
Substituting this expression into the radicand gives
\begin{align*}
(k-1)^2+(2k-1)q
&=(k-1)^2+m-(2k-1)(k-1)\\
&=m-k(k-1).
\end{align*}
Therefore
\[
\lambda(G_q)=(k-1)+\sqrt{m-k(k-1)},
\]
as claimed.
\end{proof}

\begin{lemma}[Hub-Clique Lemma]\label{lem:hub}
Let $S$ be a clique of size $\ell$, let $A'$ be a set disjoint from $S$, and let $h_1,\dots,h_j$ ($0\le j\le\ell$) be distinct vertices not in $S\cup A'$. Suppose:
\begin{itemize}
\item[(a)] every vertex of $A'$ is adjacent to every vertex of $S$;
\item[(b)] every $h_i$ is adjacent to every vertex of $S$ and to every vertex of $A'$
\end{itemize}
(no assumption is made about edges within $A'$ or among the $h_i$). If
\[
|A'|\ \ge\ j+\Big\lfloor\frac{\ell-j}{2}\Big\rfloor,
\]
then $S\cup A'\cup\{h_1,\dots,h_j\}$ contains at least $j+\big\lfloor(\ell-j)/2\big\rfloor$ pairwise vertex-disjoint triangles.
\end{lemma}

\begin{proof}
The purpose of the lemma is to turn the common-neighbor configuration
in \textup{(a)} and \textup{(b)} into an explicit triangle packing. The
idea is to use each hub with one clique vertex and one vertex of $A'$.
After that, we pair as many of the unused clique vertices as possible
and complete each pair to a triangle with another vertex of $A'$.

Put
\[
a:=\left\lfloor\frac{\ell-j}{2}\right\rfloor.
\]
Because $0\le j\le\ell$, we may first choose $j$ distinct vertices
$c_1,\ldots,c_j$ of $S$, assigning $c_i$ to the hub $h_i$. There are
$\ell-j$ clique vertices left. By the definition of $a$, we may write
\[
\ell-j=2a+\varepsilon,
\qquad \varepsilon\in\{0,1\}.
\]
Thus $2a$ of the remaining vertices can be divided into $a$ disjoint
pairs. Denote them by
\[
\{c_1',c_1''\},\ldots,\{c_a',c_a''\}.
\]
If $\varepsilon=1$, exactly one vertex of $S$ remains unused. In
particular, all the vertices
\[
c_1,\ldots,c_j,c_1',c_1'',\ldots,c_a',c_a''
\]
are distinct. This also explains the floor in the statement: once one
clique vertex has been reserved for each hub, every further triangle
requires two of the remaining clique vertices.

The hypothesis on $|A'|$ says precisely that $A'$ contains at least
$j+a$ vertices. Choose pairwise distinct vertices
\[
x_1,\ldots,x_j,y_1,\ldots,y_a\in A'.
\]
The vertices $x_1,\ldots,x_j$ will be used with the hubs, while the
vertices $y_1,\ldots,y_a$ will complete the clique pairs. Define
\[
T_i:=\{h_i,c_i,x_i\}\quad(1\le i\le j),
\qquad
U_t:=\{c_t',c_t'',y_t\}\quad(1\le t\le a).
\]
We verify the edges in these triples. For each $i$, assumption
\textup{(b)} gives $h_ic_i,h_ix_i\in E(G)$, because $c_i\in S$ and
$x_i\in A'$. Assumption \textup{(a)} gives $c_ix_i\in E(G)$. Hence
$T_i$ is a triangle. For each $t$, the edge $c_t'c_t''$ is present
because $S$ is a clique, and assumption \textup{(a)} gives
$y_tc_t',y_tc_t''\in E(G)$. Hence $U_t$ is also a triangle. Notice
that this verification uses no edge within $A'$ and no edge between two
hubs, which is why no such edges are required in the hypotheses.

It remains to check that the triangles are mutually disjoint. Two
different triangles $T_i$ and $T_{i'}$ use distinct hubs, distinct
assigned clique vertices, and distinct vertices of $A'$. Two different
triangles $U_t$ and $U_{t'}$ use disjoint clique pairs and distinct
vertices of $A'$. Finally, $T_i$ and $U_t$ are disjoint because $c_i$
does not belong to any of the chosen pairs, $x_i\ne y_t$, and the hub
$h_i$ lies outside $S\cup A'$. Thus all the constructed triangles are
pairwise vertex-disjoint.

There are $j$ triangles of the first kind and $a$ of the second kind,
so their total number is
\[
j+a=j+\left\lfloor\frac{\ell-j}{2}\right\rfloor
\]
as required. When $j=0$, the family $T_i$ is empty, and when $a=0$,
the family $U_t$ is empty; hence the same argument also covers the
boundary cases. This proves the lemma.
\end{proof}

\begin{remark}\label{rem:threshold}
Solving $j+\lfloor(\ell-j)/2\rfloor\ge k$ for integers $0\le j\le\ell$:
\[
j+\Big\lfloor\frac{\ell-j}{2}\Big\rfloor\ge k \iff \ell\ge 2k-j.
\]
In particular, provided that the corresponding lower bound on $|A'|$
in Lemma~\ref{lem:hub} is satisfied, one hub with $\ell\ge2k-1$, or
two hubs with $\ell\ge2k-2$, forces at least $k$ disjoint triangles.
\end{remark}

\section{The inductive reduction}

Throughout this section fix $k\ge3$ and put $r:=k-1$. We assume that the target spectral bound is valid at level $k-1$: for $k=3$ this is the published $2K_3$ theorem \cite{Wang2026}, while for $k\ge4$ it is the preceding inductive level of Theorem~\ref{thm:main}. Let $m\ge M_0(k)$ and let $G$ be a $kK_3$-free graph with $e(G)=m$ of maximum spectral radius among all such graphs. Since isolated vertices affect neither $m$, $\lambda$, nor $kK_3$-freeness, we assume throughout that $G$ has no isolated vertices.
Such a maximizer exists: after isolated vertices are deleted, every
$m$-edge graph has at most $2m$ vertices, so only finitely many graphs
need be considered.

We collect all size requirements in one place. Set $M_0(2)=18$, as in
the theorem of Wang, Jia and Ni. Once $M_0(k-1)$ has been fixed, let
$M_{\rm cone}(k-1)$ be the threshold in the cone-exclusion claim below.
Let $M_{\rm dens}(k)$, $M_{\rm out}(k)$, and $M_{\rm core}(k)$ be
thresholds large enough for the density reduction, the bounded-outer-layer
argument, and the finite-core exclusion, respectively. Finally, choose
$M_{\rm zero}(k)$ large enough to cover both the zero-defect equality
argument and the final elimination of $B$ in Lemma~\ref{lem:Bempty}.
We take
\[
M_0(k):=\max\bigl\{M_0(k-1),M_{\rm cone}(k-1),M_{\rm dens}(k),
M_{\rm out}(k),M_{\rm core}(k),M_{\rm zero}(k),k(k-1)+1,
(2k-1)(k-1)\bigr\}.
\]
Each auxiliary threshold is obtained in its indicated proof from finitely
many inequalities involving only $k$ (and, for the cone claim, the already
fixed $M_0(k-1)$). In particular, the recursion is well founded: in
general $M_{\rm cone}(s)$ depends only on $s$ and $M_0(s)$, never on a
later threshold. The restriction to sufficiently large size is essential;
for example, $K_{2,3}$ is $3K_3$-free, has $m=6$ and spectral radius
$\sqrt6>2=2+\sqrt{m-6}$.

If $(2k-1)\mid m$, Lemma~\ref{lem:target} exhibits a $kK_3$-free graph with $m$ edges and $\lambda=(k-1)+\sqrt{m-k(k-1)}$, so the maximality of $\lambda(G)$ forces $\lambda(G)\ge(k-1)+\sqrt{m-k(k-1)}$. To prove the target upper bound it suffices to consider the hypothesis
\begin{equation}
\lambda(G)\ \ge\ (k-1)+\sqrt{m-k(k-1)}.\label{eq:star}
\end{equation}
since if \eqref{eq:star} fails, the target bound already holds strictly for $G$ and equality is impossible. We assume \eqref{eq:star} for the remainder of this section.

In the next Lemma, we prove the connectivity of the extremal graph.

\begin{lemma}\label{lem:connected}
	Let $k\ge2$, let $m\ge M_0(k)$, and let $G$ be a $kK_3$-free graph with $e(G)=m$ edges, no isolated vertices, and maximum spectral radius among all such graphs. Then $G$ is connected.
\end{lemma}

\begin{proof}
Suppose that $G$ is disconnected. Since the adjacency matrix of $G$
is block diagonal with one block for each component,
\[
\lambda(G)=\max\{\lambda(H):H\text{ is a component of }G\}.
\]
Choose a component $G_1$ satisfying $\lambda(G_1)=\lambda(G)$.
Because $G$ has no isolated vertices and is disconnected, it has
another component $G_2$ containing an edge. Choose
\[
u\in V(G_1)\qquad\text{and}\qquad vw\in E(G_2),
\]
and form
\[
G':=G-vw+uv.
\]
The graph $G'$ still has $m$ edges. Deleting $vw$ cannot create a new
triangle, while $uv$ is the only edge joining the original component
$G_1$ to vertices outside $G_1$. Hence $uv$ is a bridge and lies on no
triangle. It follows that every triangle of $G'$ was already a triangle
of $G$, so $G'$ remains $kK_3$-free. If deleting $vw$ creates an
isolated vertex, we delete that isolated vertex as well; this changes
neither the edge count nor the spectral radius.

Let $z$ be a unit Perron vector of $G_1$, extended by zero to all other
vertices of $G'$. Both endpoints of the deleted edge $vw$ have
$z$-coordinate zero. The new edge $uv$ also contributes zero to the
Rayleigh quotient because $z_v=0$. Therefore
\begin{align*}
z^\top A(G')z
&=z^\top A(G)z\\
&=z^\top A(G_1)z
=\lambda(G_1).
\end{align*}
By the variational characterization of the largest eigenvalue,
\[
\lambda(G')\ge\lambda(G_1)=\lambda(G).
\]

We claim that this inequality is strict. Suppose, to the contrary, that
\[
\lambda(G')=\lambda(G_1).
\]
We already know that
\[
z^\top A(G')z=\lambda(G_1)=\lambda(G').
\]
Thus the unit vector $z$ attains the maximum possible Rayleigh
quotient of the real symmetric matrix $A(G')$. We recall why this
forces $z$ to be an eigenvector. Let
\[
\lambda_1=\lambda(G')\ge\lambda_2\ge\cdots\ge\lambda_n
\]
be the eigenvalues of $A(G')$, with corresponding orthonormal
eigenvectors $\phi_1,\ldots,\phi_n$, and write
$z=\sum_i\alpha_i\phi_i$. Then
\[
z^\top A(G')z=\sum_i\lambda_i\alpha_i^2
\le\lambda_1\sum_i\alpha_i^2=\lambda_1.
\]
Equality is possible only if $\alpha_i=0$ whenever
$\lambda_i<\lambda_1$. Hence $z$ lies entirely in the eigenspace
corresponding to $\lambda_1$, and therefore
\[
A(G')z=\lambda(G')z=\lambda(G_1)z.
\]

We now examine this purported eigenequation at the vertex $v$. Since
$v$ belonged to the component $G_2$, every old neighbor of $v$ lies
outside $G_1$ and therefore has $z$-coordinate zero. The only new
neighbor of $v$ is $u\in V(G_1)$. Moreover, $z_u>0$ because $z$ is a
Perron vector of the connected graph $G_1$. Thus
\[
(A(G')z)_v=z_u>0=\lambda(G_1)z_v,
\]
where the last equality uses $z_v=0$. This contradicts the
eigenequation at $v$. Thus
$\lambda(G')>\lambda(G)$, which contradicts the extremality of $G$.
Therefore $G$ is connected.
\end{proof}

Let $x$ be the Perron vector of $G$, let $u^*$ be a vertex of maximum coordinate, and set $A=N_G(u^*)$, $B=V(G)\setminus(A\cup\{u^*\})$. Counting edges by their intersection with $\{u^*\},A,B$ gives
\begin{equation}
m=|A|+e(A)+e(A,B)+e(B),\label{eq:mdecomp}
\end{equation}
and applying the eigenequation twice (once at $u^*$, once summed over $A$) gives
\begin{equation}
\lambda^2x_{u^*}=|A|x_{u^*}+\sum_{u\in A}d_A(u)x_u+\sum_{w\in B}d_A(w)x_w.\label{eq:lambda2}
\end{equation}
By Lemma~\ref{lem:target}, the equality graph satisfies $\lambda^2-2(k-1)\lambda=m-(k-1)(2k-1)$ exactly, and since $t\mapsto t^2-2(k-1)t$ is increasing for $t>k-1$, hypothesis \eqref{eq:star} gives
\begin{equation}
\big(m-(k-1)(2k-1)\big)x_{u^*}\ \le\ \big(\lambda^2-2(k-1)\lambda\big)x_{u^*}.\label{eq:ineq3}
\end{equation}
Let $A^+=\{u\in A:d_A(u)>0\}$ and $A^{0}=A\setminus A^+$.
We now give the details leading to the next inequality. The eigenequation
at $u^*$ is
\[
\lambda x_{u^*}=\sum_{u\in A}x_u.
\]
Combining this identity with \eqref{eq:lambda2}, we obtain
\begin{align*}
\bigl(\lambda^2-2(k-1)\lambda\bigr)x_{u^*}
&=|A|x_{u^*}
 +\sum_{u\in A}\bigl(d_A(u)-2(k-1)\bigr)x_u
 +\sum_{w\in B}d_A(w)x_w.
\end{align*}
Therefore \eqref{eq:ineq3} implies
\begin{align*}
\bigl(m-(k-1)(2k-1)\bigr)x_{u^*}
&\le |A|x_{u^*}
 +\sum_{u\in A}\bigl(d_A(u)-2(k-1)\bigr)x_u\\
&\qquad +\sum_{w\in B}d_A(w)x_w.
\end{align*}
By the edge decomposition \eqref{eq:mdecomp},
\[
m-|A|=e(A)+e(A,B)+e(B).
\]
Moving $|A|x_{u^*}$ to the left, using this identity, and then solving
for $e(B)x_{u^*}$ gives
\begin{align*}
e(B)x_{u^*}
&\le \sum_{u\in A}\bigl(d_A(u)-2(k-1)\bigr)x_u
 +\sum_{w\in B}d_A(w)x_w\\
&\qquad-\bigl(e(A)+e(A,B)\bigr)x_{u^*}
 +(k-1)(2k-1)x_{u^*}.
\end{align*}
Every edge between $A$ and $B$ is counted once by its endpoint in $B$,
so
\[
e(A,B)=\sum_{w\in B}d_A(w).
\]
After division by the positive number $x_{u^*}$, the terms involving
$B$ can therefore be rewritten as
\[
\sum_{w\in B}d_A(w)\frac{x_w}{x_{u^*}}-e(A,B)
=-\sum_{w\in B}d_A(w)
 \left(1-\frac{x_w}{x_{u^*}}\right).
\]
Finally, if $u\in A^0$, then $d_A(u)=0$. Hence
\begin{align*}
\sum_{u\in A}\bigl(d_A(u)-2(k-1)\bigr)\frac{x_u}{x_{u^*}}
&=\sum_{u\in A^+}\bigl(d_A(u)-2(k-1)\bigr)
  \frac{x_u}{x_{u^*}}\\
&\qquad-2(k-1)\sum_{u\in A^0}\frac{x_u}{x_{u^*}}.
\end{align*}
Substituting these two identities into the preceding bound yields
\begin{equation}
e(B)\ \le\ \sum_{u\in A^+}\big(d_A(u)-2(k-1)\big)\frac{x_u}{x_{u^*}}-e(A)-2(k-1)\sum_{u\in A^{0}}\frac{x_u}{x_{u^*}}-\sum_{w\in B}d_A(w)\Big(1-\frac{x_w}{x_{u^*}}\Big)+(k-1)(2k-1),\label{eq:master}
\end{equation}
which we use repeatedly below to constrain the structure of $A$ and $B$.

In the next lemma, we show that the edges of $A$ are numerous.

\begin{lemma}\label{lem:eAlarge}
In the setting above, $e(A)\to\infty$ as $m\to\infty$. Equivalently,
for every fixed $E_0=E_0(k)$, one has $e(A)\ge E_0$ for all sufficiently
large $m$.
\end{lemma}

\begin{proof}
Fix a constant $E_0=E_0(k)$ and suppose, for a contradiction, that
\[
e(A)\le E_0-1.
\]
Divide \eqref{eq:lambda2} by $x_{u^*}$. Since $u^*$ was chosen with
maximum Perron coordinate,
\[
0<\frac{x_v}{x_{u^*}}\le1
\qquad\text{for every }v\in V(G).
\]
It follows that
\begin{align*}
\lambda^2
&=|A|
  +\sum_{u\in A}d_A(u)\frac{x_u}{x_{u^*}}
  +\sum_{w\in B}d_A(w)\frac{x_w}{x_{u^*}}\\
&\le |A|+\sum_{u\in A}d_A(u)+\sum_{w\in B}d_A(w)\\
&=|A|+2e(A)+e(A,B).
\end{align*}
Using the edge decomposition \eqref{eq:mdecomp}, we can rewrite the
last expression as
\[
|A|+2e(A)+e(A,B)=m+e(A)-e(B).
\]
Since $e(B)\ge0$ and $e(A)\le E_0-1$, we obtain
\begin{equation}
\lambda^2\le m+E_0-1.                                             \label{eq:eA-upper}
\end{equation}

On the other hand, squaring the lower bound \eqref{eq:star} gives
\begin{align*}
\lambda^2
&\ge
\left((k-1)+\sqrt{m-k(k-1)}\right)^2\\
&=m-(k-1)+2(k-1)\sqrt{m-k(k-1)}.
\end{align*}
For fixed $k$, the quantity
\[
2(k-1)\sqrt{m-k(k-1)}-(k-1)
\]
tends to infinity with $m$. Hence, for all sufficiently large $m$,
the preceding lower bound is strictly larger than $m+E_0-1$,
contradicting \eqref{eq:eA-upper}. Therefore $e(A)\ge E_0$ for every
fixed $E_0$ once $m$ is sufficiently large. Equivalently,
$e(A)\to\infty$ as $m\to\infty$.
\end{proof}

We next show directly, using the inductive hypothesis only through a
cone estimate, that the required triangle packing can be chosen away from
$u^*$.

\begin{lemma}\label{lem:avoidustar}
For $m\ge M_0(k)$, the graph $G-u^*$ contains $k-1$ pairwise vertex-disjoint triangles.
\end{lemma}

\begin{proof}
We first establish an auxiliary estimate for cones. Its size threshold
is included in the choice made at the start of this section.

\medskip
\noindent\textbf{Claim (cone exclusion).}
Let $s\ge2$, and assume that the corresponding spectral bound is valid
at level $s$ (Theorem~\ref{thm:main} with $s$ in place of $k$). There
exists a constant $M_{\mathrm{cone}}(s)$ such
that the following holds. If $H$ is $sK_3$-free, $J=K_1\vee H$, and
$m=e(J)\ge M_{\mathrm{cone}}(s)$, then
\[
  \lambda(J)
  < s+\sqrt{m-s(s+1)}.
\]

\smallskip
\noindent\emph{Proof of the claim.}
Write
\[
  n:=|V(H)|,\qquad e:=e(H),\qquad m=n+e,
\]
and put
\[
  \rho:=\lambda(J),\qquad \mu:=\lambda(H).
\]
Let $M$ be the adjacency matrix of $H$, and let $\mathbf 1$ denote the
all-ones vector in $\mathbb R^n$. Since $J$ is connected and $H$ is a
proper induced subgraph of $J$, Perron--Frobenius gives $\rho>\mu$.
Writing a Perron vector of $J$ in the form $(a,y)^\top$, where $a$ is
the coordinate at the apex of the cone, the eigenvalue equations give
\[
  y=a(\rho I-M)^{-1}\mathbf 1
  \quad\text{and}\quad
  \rho a=\mathbf 1^\top y.
\]
Consequently,
\begin{equation}\label{eq:coronal-identity}
  \rho=\mathbf 1^\top(\rho I-M)^{-1}\mathbf 1.
\end{equation}

Let $\theta_1,\dots,\theta_n$ be the eigenvalues of $M$, with an
orthonormal eigenbasis $v_1,\dots,v_n$, and set
$c_i:=(\mathbf 1^\top v_i)^2$. Since $|\theta_i|\le\mu<\rho$, we have
\[
  \frac{1}{\rho-\theta_i}
  =\frac{\rho+\theta_i}{\rho^2-\theta_i^2}
  \le
  \frac{\rho+\theta_i}{\rho^2-\mu^2}.
\]
Using \eqref{eq:coronal-identity}, together with
\[
  \sum_{i=1}^n c_i=n
  \quad\text{and}\quad
  \sum_{i=1}^n \theta_i c_i
  =\mathbf 1^\top M\mathbf 1=2e,
\]
we obtain
\begin{equation}\label{eq:cone-coronal-bound}
  \rho
  \le \frac{\rho n+2e}{\rho^2-\mu^2},
  \qquad\text{and hence}\qquad
  \rho^3\le \rho(n+\mu^2)+2e.
\end{equation}

Set
\[
  E_s:=\max\bigl\{M_0(s),\,s(s-1),\,(s-1)(2s-1),\,1\bigr\}.
\]
If $e<E_s$, then $\mu^2\le 2e<2E_s$, and, using $n\le m$,
\eqref{eq:cone-coronal-bound} yields
\[
  \rho^2
  \le n+\mu^2+\frac{2e}\rho
  \ < m+2E_s+\frac{2E_s}{\rho}.
\]
On the other hand, for
\[
  \rho_0:=s+\sqrt{m-s(s+1)}
\]
we have
\[
  \rho_0^2-m
  =2s\sqrt{m-s(s+1)}-s\longrightarrow\infty.
\]
Thus, for all sufficiently large $m$, the inequality
$\rho\ge\rho_0$ is impossible in the case $e<E_s$.

We may therefore assume that $e\ge E_s$. Since $H$ is $sK_3$-free,
the assumed level-$s$ bound gives
\[
  \mu\le (s-1)+\sqrt{e-s(s-1)}.
\]
It follows that
\begin{equation}\label{eq:mu-polynomial-bound}
  \mu^2-2(s-1)\mu
  \le e-(s-1)(2s-1).
\end{equation}
Indeed, if $\mu\ge s-1$, then the function
$t\mapsto t^2-2(s-1)t$ is increasing on $[s-1,\infty)$, and one
evaluates it at the displayed upper bound for $\mu$; if $\mu<s-1$,
then the left-hand side of \eqref{eq:mu-polynomial-bound} is
nonpositive, whereas its right-hand side is nonnegative by the
definition of $E_s$. Since $\mu<\rho$,
\eqref{eq:mu-polynomial-bound} implies
\[
  n+\mu^2
  \le m+2(s-1)\rho-(s-1)(2s-1).
\]
Substitution into \eqref{eq:cone-coronal-bound}, followed by the
trivial estimate $e\le m$, gives
\begin{equation}\label{eq:P-nonpositive}
  P_{s,m}(\rho)\le 0,
\end{equation}
where
\[
  P_{s,m}(t)
  :=t^3-2(s-1)t^2-mt+(s-1)(2s-1)t-2m.
\]

The number $\rho_0=s+\sqrt{m-s(s+1)}$ satisfies
\begin{equation}\label{eq:rho0-quadratic}
  \rho_0^2-2s\rho_0=m-s(2s+1).
\end{equation}
Using \eqref{eq:rho0-quadratic}, a direct calculation gives
\begin{equation}\label{eq:P-rho0}
  P_{s,m}(\rho_0)
  =\rho_0-2s(2s+1).
\end{equation}
Moreover,
\[
  P_{s,m}'(t)
  =3t^2-4(s-1)t-m+(s-1)(2s-1),
\]
and another use of \eqref{eq:rho0-quadratic} yields
\[
  P_{s,m}'(\rho_0)
  =2\rho_0^2+(4-2s)\rho_0-4s+1>0
\]
for all sufficiently large $m$. Since
$P_{s,m}''(t)=6t-4(s-1)>0$ for every $t\ge\rho_0$ once $m$ is large,
$P_{s,m}$ is strictly increasing on $[\rho_0,\infty)$. Enlarging
$M_{\mathrm{cone}}(s)$ once more, we may also assume that
$\rho_0>2s(2s+1)$. Therefore \eqref{eq:P-rho0} gives
\[
  P_{s,m}(t)\ge P_{s,m}(\rho_0)>0
  \qquad (t\ge\rho_0),
\]
which contradicts \eqref{eq:P-nonpositive} whenever $\rho\ge\rho_0$.
This proves the claim.
\hfill$\triangleleft$

\medskip
We now prove the lemma. Suppose, to the contrary, that $G-u^*$ is
$(k-1)K_3$-free. We first show that $B=\varnothing$. Assume that
$w\in B$. Since $G$ has no isolated vertices, choose
$z\in N_G(w)$. By the maximal choice of the Perron coordinate at
$u^*$, we have $x_{u^*}\ge x_z$, while $u^*w\notin E(G)$. Define
\[
  G':=G-wz+u^*w.
\]
Lemma~\ref{lem:shift} gives $\lambda(G')>\lambda(G)$.

We claim that $G'$ is still $kK_3$-free. Otherwise, let
$T_1',\dots,T_k'$ be pairwise vertex-disjoint triangles in $G'$. If
none of them uses the new edge $u^*w$, then all of them are already
present in $G$, contradicting the fact that $G$ is $kK_3$-free. If
one of them uses $u^*w$, then the remaining $k-1$ triangles avoid
$u^*$ and use only edges of $G$; hence they form $(k-1)$ pairwise
vertex-disjoint triangles in $G-u^*$, contrary to our assumption.
Thus $G'$ is $kK_3$-free. Since $e(G')=e(G)=m$, this contradicts the
extremality of $G$. Therefore $B=\varnothing$.

It follows that $u^*$ is adjacent to every other vertex of $G$, and
hence
\[
  G=K_1\vee (G-u^*).
\]
Apply the cone-exclusion claim with $s=k-1$ and $H=G-u^*$. By the
choice $M_0(k)\ge M_{\mathrm{cone}}(k-1)$, we
obtain
\[
  \lambda(G)
  <(k-1)+\sqrt{m-k(k-1)},
\]
contradicting hypothesis \eqref{eq:star}. Hence $G-u^*$ contains
$k-1$ pairwise vertex-disjoint triangles, as required.
\end{proof}

\begin{lemma}\label{lem:cover2}
There is a set $S\subseteq A$ with $|S|\le3(k-1)$, arising from $k-1$ pairwise disjoint triangles $T_1,\dots,T_{k-1}\subseteq G-u^*$, such that every edge of $G[A]$ meets $S$.
\end{lemma}

\begin{proof}
By Lemma~\ref{lem:avoidustar}, the graph $G-u^*$ contains
$k-1$ pairwise vertex-disjoint triangles. Fix such a collection and
write it as
\[
T_1,\ldots,T_{k-1}.
\]
Let
\[
T:=\bigcup_{i=1}^{k-1}V(T_i)
\qquad\text{and}\qquad
S:=A\cap T.
\]
Since the triangles are pairwise vertex-disjoint and each has three
vertices,
\[
|T|=3(k-1).
\]
Consequently, $S\subseteq A$ and
\[
|S|\le |T|=3(k-1).
\]

It remains to show that $S$ meets every edge of $G[A]$. Suppose, to
the contrary, that an edge $ab\in E(G[A])$ has neither endpoint in
$S$. Since $a,b\in A$ and $S=A\cap T$, this implies
\[
\{a,b\}\cap T=\varnothing.
\]
By the definition $A=N_G(u^*)$, both $a$ and $b$ are adjacent to
$u^*$. Together with the edge $ab$, these two edges show that
\[
\{u^*,a,b\}
\]
is a triangle in $G$.

This new triangle is disjoint from every $T_i$. Indeed,
$u^*\notin T$ because each $T_i$ lies in $G-u^*$, while
$a,b\notin T$ by the preceding observation. Therefore
\[
\{u^*,a,b\},T_1,\ldots,T_{k-1}
\]
are $k$ pairwise vertex-disjoint triangles in $G$. This contradicts
the assumption that $G$ is $kK_3$-free. Hence every edge of $G[A]$
meets $S$, completing the proof.
\end{proof}

For orientation, fix the triangles $T_1,\ldots,T_{k-1}$ supplied by
Lemma~\ref{lem:cover2}, write
\[
T:=\bigcup_{i=1}^{k-1}V(T_i),
\qquad S_0:=A\cap T,
\]
and let $p_0:=|S_0|\le3(k-1)$. Thus $S_0$ contains precisely those
vertices of the chosen triangles that lie in the neighborhood
$A=N_G(u^*)$. Set
\[
A_{\rm out}:=A\setminus S_0.
\]
Lemma~\ref{lem:cover2} says that $S_0$ meets every edge of $G[A]$.
Consequently, two vertices of $A_{\rm out}$ cannot be adjacent, since
an edge between them would avoid $S_0$. Hence $A_{\rm out}$ is
independent in $G[A]$. The same observation shows that every neighbor
in $A$ of a vertex $v\in A_{\rm out}$ must belong to $S_0$; that is,
$N_A(v)\subseteq S_0$. Therefore every edge of $G[A]$ either lies
inside $S_0$ or joins $S_0$ to $A_{\rm out}$, and so
\[
e(A)=e(S_0)+e(S_0,A_{\rm out}).
\]
Because $k$ is fixed, the size $p_0$ is bounded independently of $m$,
and hence $e(S_0)\le\binom{p_0}{2}=O_k(1)$. On the other hand,
Lemma~\ref{lem:eAlarge} gives $e(A)\to\infty$ as $m\to\infty$.
It follows that $e(S_0,A_{\rm out})\to\infty$. Since each vertex of
$A_{\rm out}$ has at most $p_0$ neighbors in $S_0$, we also have
\[
e(S_0,A_{\rm out})\le p_0|A_{\rm out}|,
\]
and therefore $|A_{\rm out}|\to\infty$. Thus $S_0$ is a bounded vertex
cover of $G[A]$ attached to a growing independent set. This observation
is useful, but it does not yet identify which vertices form the true
clique core, as the following remark explains.

\begin{remark}\label{rem:S0-not-clique}
The covering set $S_0$ need not be the final clique core. To see this,
consider the target graph, with $s\ge k-1$,
\[
G^\star=K_{2k-1}\vee sK_1
\]
and choose $u^*$ in its clique. Then $B=\varnothing$, and
$G^\star-u^*$ has $2k-2$ remaining clique vertices. Since the other
part is independent, every triangle in $G^\star-u^*$ contains at
least two of these clique vertices.

Consequently, any collection of $k-1$ disjoint triangles in
$G^\star-u^*$ must use all $2k-2$ remaining clique vertices, exactly
two in each triangle. Each triangle also uses one distinct vertex from
the independent set. Thus
\[
|T|=3(k-1),\qquad S_0=A\cap T=T,
\]
and hence $p_0=3(k-1)$ rather than $2(k-1)$. Moreover,
\[
G^\star[S_0]\cong K_{2k-2}\vee (k-1)K_1,
\]
so $S_0$ is not a clique when $k\ge3$.

The $k-1$ independent vertices that happen to lie in $S_0$ are no
different from the other independent vertices of $G^\star$: every
independent vertex has $A$-degree $2k-2$. In contrast, each genuine
core vertex has $A$-degree $2k-3+s$, which grows with $s$. Therefore
membership in $S_0$ does not identify the clique core; the useful
distinction is whether the $A$-degree remains bounded or grows with
$m$.
\end{remark}

Motivated by Remark~\ref{rem:S0-not-clique}, set $r=k-1$ and fix the
degree threshold
\[
N_{\rm core}:=\binom{3r}{2}+1.
\]
The next lemma includes all of the resulting data in its statement so that
the central reduction can be read independently of this setup.

\begin{lemma}[Density reduction]\label{lem:density}
Let $T_1,\ldots,T_r\subseteq G-u^*$ be the pairwise vertex-disjoint
triangles supplied by Lemma~\ref{lem:avoidustar}, put
\[
T:=\bigcup_{i=1}^rV(T_i),\qquad S_0:=A\cap T,
\]
and note that every edge of $G[A]$ meets $S_0$ by
Lemma~\ref{lem:cover2}. With $N_{\rm core}=\binom{3r}{2}+1$, define
\[
S_{\rm hi}:=\{v\in S_0:d_A(v)\ge N_{\rm core}\},\qquad
R_0:=S_0\setminus S_{\rm hi},\qquad
U:=A\setminus S_{\rm hi},\qquad \sigma:=|S_{\rm hi}|.
\]
Finally, let
\[
C:=\{v\in A:d_A(v)>2r\},\qquad L:=A\setminus C,
\qquad c:=|C|,\qquad q:=c-2r,
\]
and normalize the Perron vector by $y_v:=x_v/x_{u^*}$. For all sufficiently large $m$ the following hold.
\begin{enumerate}[label=\textup{(\roman*)}]
\item $\sigma\ge2r$ and $2r\le c\le4r$.
\item The exact defect inequality
\begin{align}
&\left[\binom c2-e(C)\right]+e(L)+e(B)
+\sum_{v\in C}(d_A(v)-2r)(1-y_v)\notag\\
&\qquad
+\sum_{v\in L}(2r-d_A(v))y_v
+\sum_{w\in B}d_A(w)(1-y_w)
\ \le\ \binom q2                                                   \label{eq:defect}
\end{align}
holds; every term on its left is nonnegative.
\item If $q=0$, then
\[
G[C]\cong K_{2r},\qquad e(L)=e(B)=0,
\qquad d_A(v)=2r\quad(v\in L),
\]
$L$ is complete to $C$, $y_v=1$ for every $v\in C$, and $y_w=1$ for every $w\in B$ with $d_A(w)>0$.
\item If $q\ge1$, then
\begin{equation}
\binom q2\ge2r-1.                                                   \label{eq:qgap}
\end{equation}
In particular, since $r\ge2$, neither $q=1$ nor $q=2$ can occur.
\end{enumerate}
\end{lemma}

\begin{proof}
The purpose of this lemma is to turn the spectral hypothesis into the
exact nonnegative defect inequality \eqref{eq:defect}, which will control
all subsequent structural errors.
Put $E_U:=e(G[U])$. Every edge of $G[A]$ meets
$S_0=S_{\rm hi}\cup R_0$, so every edge of $G[U]$ meets $R_0$.
Consequently
\begin{equation}
E_U\le\sum_{v\in R_0}d_A(v)\le |R_0|(N_{\rm core}-1)
\le3r(N_{\rm core}-1)<3rN_{\rm core}=:C_0.
\end{equation}
The eigenequation at $u^*$ gives $\sum_{a\in A}y_a=\lambda$. Applying it twice at $u^*$ and subtracting $\sigma\lambda$ gives
\begin{equation}
\lambda^2-\sigma\lambda
=|A|+\sum_{a\in A}(d_A(a)-\sigma)y_a+\sum_{w\in B}d_A(w)y_w.     \label{eq:sigmaidentity}
\end{equation}
Because $N_{\rm core}>3r\ge\sigma$, the coefficient
$d_A(v)-\sigma$ is positive for $v\in S_{\rm hi}$, and hence
\[
\sum_{v\in S_{\rm hi}}(d_A(v)-\sigma)y_v
\le\sum_{v\in S_{\rm hi}}(d_A(v)-\sigma).
\]
For $a\in U$, the inequality $d_A(a)\le\sigma+d_U(a)$ gives
$(d_A(a)-\sigma)y_a\le d_U(a)$. Therefore
\begin{align*}
\sum_{a\in A}(d_A(a)-\sigma)y_a
&\le2e(S_{\rm hi})+e(S_{\rm hi},U)-\sigma^2+2E_U\\
&=e(A)+e(S_{\rm hi})-\sigma^2+E_U\le e(A)+C_0.
\end{align*}
Also $\sum_{w\in B}d_A(w)y_w\le e(A,B)$. Using
$m=|A|+e(A)+e(A,B)+e(B)$ in \eqref{eq:sigmaidentity}, we obtain
\begin{equation}
\lambda^2-\sigma\lambda\le m+C_0.                                 \label{eq:sigmaupper}
\end{equation}
On the other hand, hypothesis \eqref{eq:star} gives
\begin{equation}
\lambda^2-2r\lambda\ge m-r(2r+1).                                 \label{eq:quadraticlower}
\end{equation}
If $\sigma\le2r-1$, then \eqref{eq:quadraticlower} yields
\[
\lambda^2-\sigma\lambda
\ge m-r(2r+1)+(2r-\sigma)\lambda>m+C_0
\]
for sufficiently large $m$, contradicting \eqref{eq:sigmaupper}. Thus
$\sigma\ge2r$. Since $N_{\rm core}>2r$, we have
$S_{\rm hi}\subseteq C$, so $c\ge2r$.

We next prove the upper bound on $c$. At most $3r$ vertices of $C$ lie in
$T:=\bigcup_{i=1}^rV(T_i)$. If $a\in C\setminus T$, then every
$A$-neighbor of $a$ lies in $S_0=A\cap T$. Since $d_A(a)>2r$, the
pigeonhole principle shows that $a$ is adjacent to all three vertices of
some $T_i$. For each fixed $i$ there is at most one such vertex $a$:
if distinct $a,b\in C\setminus T$ both see
$V(T_i)=\{v_1,v_2,v_3\}$, then
\[
\{a,v_1,v_2\},\quad\{u^*,b,v_3\},\quad T_j\ (j\ne i)
\]
are $r+1=k$ pairwise vertex-disjoint triangles, a contradiction. Hence
$|C\setminus T|\le r$ and $c\le4r$, proving (i).

We now rewrite the master inequality \eqref{eq:master}. Its two sums over
$A^+$ and $A^{0}$ combine as
\[
\sum_{a\in A}(d_A(a)-2r)y_a.
\]
For $v\in C$ use
\[
(d_A(v)-2r)y_v=(d_A(v)-2r)-(d_A(v)-2r)(1-y_v),
\]
whereas for $v\in L$ use
$(d_A(v)-2r)y_v=-(2r-d_A(v))y_v$. Moreover,
\begin{equation}
\sum_{v\in C}(d_A(v)-2r)-e(A)=e(C)-e(L)-2rc.                       \label{eq:degreecore}
\end{equation}
Indeed, $\sum_{v\in C}d_A(v)=2e(C)+e(C,L)$ and
$e(A)=e(C)+e(C,L)+e(L)$. Finally, since $c=2r+q$,
\begin{equation}
e(C)-2rc+r(2r+1)
=\binom q2-\left[\binom c2-e(C)\right].                            \label{eq:binomialidentity}
\end{equation}
Substitution of \eqref{eq:degreecore} and \eqref{eq:binomialidentity}
into \eqref{eq:master} gives \eqref{eq:defect} exactly. Its summands are
nonnegative because $y_v\le1$, proving (ii).

If $q=0$, the right side of \eqref{eq:defect} is zero. Since all six
terms on its left are nonnegative, each of them must vanish. The first
three terms give
\[
\binom{|C|}{2}-e(C)=0,\qquad e(L)=0,
\qquad e(B)=0.
\]
As $|C|=2r$, the first equality says that $G[C]\cong K_{2r}$. The fifth
term is
\[
\sum_{v\in L}(2r-d_A(v))y_v=0.
\]
For $v\in L$ we have $d_A(v)\le2r$ by the definition of $L$, while
$y_v>0$ by positivity of the Perron vector. Hence every individual term
is zero only if $d_A(v)=2r$. Since $e(L)=0$ and $|C|=2r$, all these
$2r$ neighbors lie in $C$, so every vertex of $L$ is adjacent to every
vertex of $C$.

It remains to justify the two Perron-coordinate conclusions. The fourth
term in \eqref{eq:defect} is
\[
\sum_{v\in C}(d_A(v)-2r)(1-y_v)=0.
\]
For every $v\in C$, the definition of $C$ gives $d_A(v)-2r>0$.
Moreover, $0<y_v\le1$ because $y_v=x_v/x_{u^*}$ and $u^*$ was chosen
with maximum Perron coordinate. Thus both factors are nonnegative, and
the first factor is strictly positive; consequently $1-y_v=0$, or
$y_v=1$, for every $v\in C$.

Similarly, the sixth term is
\[
\sum_{w\in B}d_A(w)(1-y_w)=0.
\]
If $w\in B$ satisfies $d_A(w)>0$, then the first factor is strictly
positive and $1-y_w\ge0$. Therefore the corresponding summand can vanish
only when $y_w=1$. This proves all the assertions in (iii).

It remains to prove (iv). Suppose $q\ge1$. Inequality
\eqref{eq:defect} gives
\begin{equation}
\sum_{\substack{v\in L\\d_A(v)<2r}}y_v\le\binom q2,
\qquad e(L)\le\binom q2.                                           \label{eq:Lmass}
\end{equation}
Since $c\le4r$, we have $q\le2r$. Also
$\sum_{a\in A}y_a=\lambda\to\infty$, while
$\sum_{v\in C}y_v\le4r$ and the first sum in \eqref{eq:Lmass} is
bounded. Hence the number of vertices $v\in L$ with $d_A(v)=2r$ tends
to infinity with $m$. Delete from this set the endpoints of the at most
$\binom q2$ edges of $G[L]$. Every remaining vertex has no neighbor in
$L$, so its $A$-neighborhood is a $2r$-subset of $C$. For large $m$,
the remaining set contains
$r$ distinct vertices $z_1,\ldots,z_r$ having a common $A$-neighborhood
$W\subseteq C$ of size $2r$; this follows by the pigeonhole principle,
as there are at most $\binom{c}{2r}\le\binom{4r}{2r}$ possible neighborhoods.

The graph $G[W]$ has no perfect matching. Otherwise, let
$a_1b_1,\ldots,a_rb_r$ be a perfect matching of $G[W]$. Then
$\{z_i,a_i,b_i\}$, $1\le i\le r$, are pairwise disjoint triangles.
Choose $c_0\in C\setminus W$, which is possible because $q\ge1$.
As $d_A(c_0)>2r=|W|$, there is
$z\in N_A(c_0)\setminus W$. None of the $z_i$ can equal $z$, because
$N_A(z_i)=W$. Thus $\{u^*,c_0,z\}$ is disjoint from the preceding $r$
triangles, producing $k=r+1$ disjoint triangles, a contradiction.

By Lemma~\ref{lem:matching}, at least $2r-1$ edges are missing from
$G[W]$. These are also missing from $G[C]$, so
\[
2r-1\le\binom c2-e(C)\le\binom q2
\]
by \eqref{eq:defect}. This is \eqref{eq:qgap} and completes the proof.
\end{proof}

\begin{lemma}[Bounded outer layer]\label{lem:Bbounded}
In the setting of Lemma~\ref{lem:density}, suppose $q\ge1$ and put
\[
D:=\binom q2.
\]
For all sufficiently large $m$,
\[
|B|\le 8D+r-1.
\]
\end{lemma}

\begin{proof}
The purpose of this lemma is to convert the weighted information in the defect inequality into a uniform bound on the number of vertices in the outer set $B$.
By \eqref{eq:defect},
\begin{equation}
e(B)\le D,\qquad
\sum_{b\in B}d_A(b)(1-y_b)\le D.                                  \label{eq:Bbudget}
\end{equation}
Let $B^\times$ be the set of endpoints of the edges in $G[B]$, and
put $B_0:=B\setminus B^\times$. Thus
\[
|B^\times|\le2D,\qquad N_B(b)=\varnothing\quad(b\in B_0).
\]
For $b\in B_0$, define its missing $A$-mass by
\[
M_b:=\sum_{a\in A\setminus N_A(b)}y_a.
\]
The eigenequations at $u^*$ and $b$ give
\[
M_b=\lambda(1-y_b).
\]
Consequently, if $\delta_b:=d_A(b)(1-y_b)$, then
\begin{equation}
\delta_b=\frac{d_A(b)M_b}{\lambda},
\qquad
d_A(b)\ge\sum_{a\in N_A(b)}y_a=\lambda-M_b.                        \label{eq:Mb}
\end{equation}
Split
\[
\mathcal H:=\{b\in B_0:M_b<\lambda/2\},\qquad
\mathcal J:=B_0\setminus\mathcal H.
\]
For $b\in\mathcal H$, \eqref{eq:Mb} gives
$\delta_b\ge M_b/2$, whereas for $b\in\mathcal J$ it gives
$\delta_b\ge d_A(b)/2$. Hence \eqref{eq:Bbudget} yields
\begin{equation}
\sum_{b\in\mathcal H}M_b\le2D,\qquad
\sum_{b\in\mathcal J}d_A(b)\le2D.                                 \label{eq:HJbounds}
\end{equation}
Every $b\in B_0$ has $d_A(b)\ge1$, because $G$ has no isolated
vertices and $B_0$ has no neighbor in $B$. Therefore
\begin{equation}
|\mathcal J|\le2D.                                                  \label{eq:Jcard}
\end{equation}

We next bound $\mathcal H$. Let
\[
L_1:=\{z\in L:d_A(z)=2r\},
\]
and let $Q_L$ be the set of endpoints of the edges in $G[L]$. From
\eqref{eq:defect},
\[
\sum_{\substack{z\in L\\d_A(z)<2r}}y_z\le D,\qquad
e(L)\le D,
\]
so $|Q_L|\le2D$. Since $\sum_{a\in A}y_a=\lambda$, $c\le4r$, and
$y_a\le1$, the set
\[
L^\circ:=L_1\setminus Q_L
\]
satisfies
\begin{equation}
\sum_{z\in L^\circ}y_z\ge\lambda-4r-3D.                            \label{eq:Lcirc-mass}
\end{equation}
Let $\mathcal Z$ be the set of vertices in $L^\circ$ adjacent to every
vertex of $\mathcal H$. By the union bound and the definition of
$M_b$,
\[
\sum_{z\in L^\circ\setminus\mathcal Z}y_z
\le\sum_{b\in\mathcal H}M_b\le2D.
\]
Thus
\begin{equation}
\sum_{z\in\mathcal Z}y_z\ge\lambda-4r-5D.                          \label{eq:Zmass}
\end{equation}
Every $z\in\mathcal Z$ has no neighbor in $L$ and has exactly $2r$
neighbors in $A$, all of which lie in $C$. Since $c\le4r$, there are
at most $\binom{4r}{2r}$ possible $A$-neighborhoods. It follows from
\eqref{eq:Zmass} that, for some $W\subseteq C$ with $|W|=2r$, the set
\[
Z_W:=\{z\in\mathcal Z:N_A(z)=W\}
\]
has cardinality tending to infinity with $m$.

The graph $G[W]$ has no perfect matching. Indeed, suppose
$a_1b_1,\ldots,a_rb_r$ is a perfect matching of $G[W]$. Choose
distinct $z_1,\ldots,z_r\in Z_W$; then
$\{a_i,b_i,z_i\}$, $1\le i\le r$, are disjoint triangles. Since
$q\ge1$, choose $c_0\in C\setminus W$. The inequality
$d_A(c_0)>2r=|W|$ gives a vertex
$z_0\in N_A(c_0)\setminus W$. No vertex of $Z_W$ can equal $z_0$,
because its $A$-neighborhood is exactly $W$. Hence
$\{u^*,c_0,z_0\}$ is a further triangle disjoint from the preceding
$r$, a contradiction.

For each $w\in W$, we have $d_A(w)\ge|Z_W|$. The fourth summand in
\eqref{eq:defect} therefore gives
\[
1-y_w\le\frac{D}{d_A(w)-2r}=o(1).
\]
In particular, $y_w\ge1/2$ for every $w\in W$ once $m$ is large.
Every missing incidence between $\mathcal H$ and $W$ contributes at
least $1/2$ to the corresponding missing mass $M_b$. By
\eqref{eq:HJbounds}, there are at most $4D$ such incidences. Thus all
but at most $4D$ vertices of $\mathcal H$ are complete to $W$; denote
the set of these vertices by $\mathcal H_0$.

Let $\alpha:=\nu(G[W])$ and $\beta:=r-\alpha$. Since $G[W]$ has no
perfect matching, $\beta\ge1$. Complete a maximum matching of $G[W]$
with $\alpha$ distinct vertices of $Z_W$ to obtain $\alpha$ disjoint
triangles. There remain $2\beta$ unmatched vertices of $W$. Using distinct vertices of
$\{u^*\}\cup\mathcal H_0$, distinct unmatched vertices of $W$, and
fresh distinct vertices of $Z_W$, we can form
\[
\min\{|\mathcal H_0|+1,2\beta\}
\]
further disjoint triangles: every required edge is present because
$u^*$ and every vertex of $\mathcal H_0$ are complete to both $W$
and $Z_W$, while $Z_W$ is complete to $W$. Since $G$ is
$(r+1)K_3$-free,
\[
\alpha+\min\{|\mathcal H_0|+1,2\beta\}\le r=\alpha+\beta.
\]
As $\beta\ge1$, this forces $|\mathcal H_0|\le\beta-1\le r-1$. Therefore
\[
|\mathcal H|\le4D+r-1.
\]
Combining this with $|B^\times|\le2D$ and \eqref{eq:Jcard} gives
\[
|B|\le |B^\times|+|\mathcal H|+|\mathcal J|
\le8D+r-1,
\]
as claimed.
\end{proof}

\begin{lemma}[Finite-core exclusion]\label{lem:finitecore}
In the setting of Lemma~\ref{lem:density}, for all sufficiently large $m$ one has
\[
q=0,\qquad\text{and hence}\qquad |C|=2r.
\]
\end{lemma}

\begin{proof}
The purpose of this lemma is to reduce the graph to a bounded core with independent twin classes and then use that finite description to rule out the case $q>0$. Suppose for a contradiction that $q\ge1$, and continue to write $D=\binom q2$. Lemma~\ref{lem:Bbounded}, together with $q\le2r$, shows that $|B|=O_r(1)$. Let $Q$ be the set of endpoints of the edges in $G[L]$. By \eqref{eq:defect}, $|Q|\le2D=O_r(1)$. Put
\[
K:=\{u^*\}\cup C\cup B\cup Q,\qquad I:=L\setminus Q,\qquad
H_K:=G[K].
\]
The set $K$ has bounded order in terms of $r$, and $I$ is independent. Moreover, every neighbor of a vertex in $I$ belongs to $K$.

Partition $I$ into its finitely many neighborhood types $P_1,\ldots,P_\tau\subseteq K$. Write $n_i$ for the multiplicity of
type $P_i$, and let $\chi_i\in\{0,1\}^{K}$ be its incidence vector. Then
\begin{equation}
m=e(H_K)+\sum_{i=1}^\tau n_i|P_i|.                                \label{eq:type-edges}
\end{equation}
Let $a$ be the restriction to $K$ of a Perron vector of $G$, and let $N$ be the $K$-by-$I$ incidence matrix. Eliminating the coordinates on the independent set $I$ from the two block eigenequations gives
\[
\lambda^2a=\lambda A(H_K)a+Ma,\qquad
M:=NN^\top=\sum_{i=1}^\tau n_i\chi_i\chi_i^\top.
\]
Put $v:=a/\|a\|_2$ and
\[
\Delta:=\operatorname{tr}M-v^\top Mv.
\]
Since $M$ is positive semidefinite, $\Delta\ge0$. Taking the inner product with $v$ and using \eqref{eq:type-edges}, we obtain the exact identity
\begin{equation}
\lambda^2=m-e(H_K)-\Delta+\lambda v^\top A(H_K)v.                  \label{eq:twin-identity}
\end{equation}

We next identify a neighborhood type of linear multiplicity. Define
\[
L^\circ:=\{z\in L:d_A(z)=2r,\ d_L(z)=0\}.
\]
We give the mass estimate explicitly. Since $c=|C|\le4r$ and $0<y_v\le1$ for every vertex,
\[
\sum_{v\in C}y_v\le4r.
\]
Moreover, the fifth term in \eqref{eq:defect} gives
\[
\sum_{v\in L}(2r-d_A(v))y_v\le D.
\]
If $v\in L$ and $d_A(v)<2r$, then the integer $2r-d_A(v)$ is at least one. Consequently
\[
\sum_{\substack{v\in L\\d_A(v)<2r}}y_v\le D.
\]
Using $\sum_{a\in A}y_a=\lambda$ and decomposing $A$ into $C$, the vertices of $L$ with $d_A(v)<2r$, and the vertices of $L$ with
$d_A(v)=2r$, we obtain
\begin{align*}
\sum_{\substack{v\in L\\d_A(v)=2r}}y_v
&=\lambda-\sum_{v\in C}y_v
-\sum_{\substack{v\in L\\d_A(v)<2r}}y_v\\
&\ge\lambda-4r-D.
\end{align*}

Recall that $Q$ is the set of endpoints of the edges in $G[L]$. The second term of \eqref{eq:defect} gives $e(L)\le D$, and hence $|Q|\le2e(L)\le2D$. A vertex of $L$ has $d_L(v)=0$ exactly when it does not belong to $Q$. Thus $L^\circ$ is obtained from $\{v\in L:d_A(v)=2r\}$ by deleting its intersection with $Q$. Since every deleted vertex has Perron ratio at most one,
\begin{align*}
\sum_{z\in L^\circ}y_z
&\ge \sum_{\substack{v\in L\\d_A(v)=2r}}y_v-|Q|\\
&\ge\lambda-4r-D-2D\\
&=\lambda-4r-3D.
\end{align*}
For $z\in L^\circ$, the eigenequation and Lemma~\ref{lem:Bbounded} give
\[
\lambda y_z
=1+\sum_{a\in N_A(z)}y_a+\sum_{b\in N_B(z)}y_b
\le1+2r+|B|=O_r(1).
\]
Together with the preceding mass estimate, this shows that $|L^\circ|=\Omega_r(\lambda^2)=\Omega_r(m)$; here we use \eqref{eq:star}, which gives $\lambda^2=\Omega_r(m)$. Each vertex of $L^\circ$ has a full neighborhood of the form
\[
P=\{u^*\}\cup W\cup R_B,\qquad
W\subseteq C,\quad |W|=2r,\quad R_B\subseteq B.
\]
There are only $O_r(1)$ such types. Hence one of them, denoted by $P$, has multiplicity
\begin{equation}
n_P\ge\eta_rm                                                        \label{eq:linear-type}
\end{equation}
for some constant $\eta_r>0$. Put $h:=|P|$ and
$\chi_P:=\mathbf1_P$.

The lower bound \eqref{eq:quadraticlower} and
\eqref{eq:twin-identity} imply
\[
\Delta
\le\lambda\big(v^\top A(H_K)v-2r\big)+r(2r+1)-e(H_K)
=O_r(\lambda),
\]
because $|K|=O_r(1)$. On the other hand,
\[
\Delta
=\sum_{i=1}^\tau n_i\bigl(|P_i|-(\chi_i^\top v)^2\bigr),
\]
and every summand is nonnegative by Cauchy--Schwarz. From
\eqref{eq:linear-type}, the standard bound $\lambda^2\le2m$, and the
preceding estimate,
\begin{equation}
h-(\chi_P^\top v)^2=O_r(\lambda^{-1}).                             \label{eq:type-concentration}
\end{equation}
Let $w:=\chi_P/\sqrt h$. Both $v$ and $w$ are nonnegative unit vectors.
Equation \eqref{eq:type-concentration} gives
$1-(w^\top v)^2=O_r(\lambda^{-1})$, and therefore
\[
\|v-w\|_2^2=2(1-w^\top v)=O_r(\lambda^{-1}).
\]
Since $H_K$ has bounded order,
\begin{equation}
v^\top A(H_K)v
=w^\top A(H_K)w+O_r(\lambda^{-1/2})
=\frac{2e(H_K[P])}{h}+O_r(\lambda^{-1/2}).                         \label{eq:core-Rayleigh}
\end{equation}

We claim that $\nu(H_K[P])\le r$. Otherwise, $r+1$ independent edges
of $H_K[P]$, completed with $r+1$ distinct vertices of the type $P$,
would give $r+1=k$ pairwise vertex-disjoint triangles in $G$.
Since $h=2r+1+|R_B|\ge2r+1$, the Erd\H{o}s--Gallai matching theorem
gives
\begin{equation}
e(H_K[P])
\le\max\left\{\binom{2r+1}{2},
\binom r2+r(h-r)\right\},
\qquad\text{so}\qquad
\frac{e(H_K[P])}{h}\le r.                                         \label{eq:EG-ratio}
\end{equation}
Equality in the last inequality is possible only when
\begin{equation}
h=2r+1\qquad\text{and}\qquad H_K[P]\cong K_{2r+1}.                \label{eq:EG-equality}
\end{equation}

If equality does not hold in \eqref{eq:EG-ratio}, then, because
$h\le|K|=O_r(1)$ and there are only finitely many possible graphs
$H_K[P]$, there is a constant $\gamma_r>0$ such that
\[
\frac{e(H_K[P])}{h}\le r-\gamma_r.
\]
Equations \eqref{eq:twin-identity} and \eqref{eq:core-Rayleigh}, with
$\Delta\ge0$ and $e(H_K)\ge0$, would then give
\[
\lambda^2
\le m+2(r-\gamma_r)\lambda+O_r(\lambda^{1/2}),
\]
contradicting
\[
\lambda^2\ge m+2r\lambda-r(2r+1)
\]
from \eqref{eq:quadraticlower} when $m$ is sufficiently large.
Thus \eqref{eq:EG-equality} holds. In particular $R_B=\varnothing$ and
$G[W]\cong K_{2r}$.

Choose a perfect matching
$a_1b_1,\ldots,a_rb_r$ of $G[W]$ and distinct vertices
$z_1,\ldots,z_r$ of the type $P$. These give the $r$ disjoint
triangles $\{a_i,b_i,z_i\}$. Since $q\ge1$, take
$c_0\in C\setminus W$. As $d_A(c_0)>2r=|W|$, there is
$z_0\in N_A(c_0)\setminus W$. The vertex $z_0$ is not one of the
$z_i$, because every $z_i$ has $A$-neighborhood exactly $W$.
Therefore $\{u^*,c_0,z_0\}$ is a further disjoint triangle, yielding
$r+1=k$ disjoint triangles in $G$, a contradiction. Hence $q=0$.
Lemma~\ref{lem:density}\textup{(i)} now gives $|C|=2r$.
\end{proof}

\begin{lemma}[Zero-defect equality]\label{lem:equality}
If $|C|=2r$, then
\[
\lambda(G)=r+\sqrt{m-r(r+1)},
\]
and the conclusions in Lemma~\ref{lem:density}\textup{(iii)} hold.
\end{lemma}

\begin{proof}
The purpose of this lemma is to show that $q=0$ forces every defect term to vanish, thereby determining both the extremal eigenvalue and the corresponding graph structure. Since $c=|C|=2r$, we have $q=c-2r=0$.  Let $\mathcal D$ denote the entire left-hand side of the defect inequality \eqref{eq:defect}.  Each of its six summands is nonnegative, while Lemma~\ref{lem:density}\textup{(ii)} gives
\[
0\le \mathcal D\le\binom q2=0.
\]
Thus $\mathcal D=0$, and in fact every one of its summands vanishes. We first spell out the resulting structural information.

The first three summands give
\[
\binom{|C|}{2}-e(C)=0,\qquad e(L)=0,\qquad e(B)=0.
\]
Hence $G[C]\cong K_{2r}$, and both $L$ and $B$ are independent.  The
fifth summand gives
\[
\sum_{v\in L}(2r-d_A(v))y_v=0.
\]
For $v\in L$, the definition of $L$ gives $d_A(v)\le2r$, and
$y_v>0$ because the Perron vector is positive.  Therefore
$d_A(v)=2r$ for every $v\in L$.  Since $e(L)=0$ and $|C|=2r$, all
of these neighbors lie in $C$; consequently every vertex of $L$ is
complete to $C$.

The fourth summand is
\[
\sum_{v\in C}(d_A(v)-2r)(1-y_v)=0.
\]
Here $d_A(v)-2r>0$ by the definition of $C$, and $1-y_v\ge0$ because
$u^*$ has maximum Perron coordinate.  It follows that $y_v=1$ for
every $v\in C$.  Similarly, the sixth summand
\[
\sum_{w\in B}d_A(w)(1-y_w)=0
\]
shows that $y_w=1$ whenever $w\in B$ and $d_A(w)>0$.  These are
precisely the conclusions of Lemma~\ref{lem:density}\textup{(iii)}.

It remains to derive the asserted value of $\lambda$.  We record
explicitly the slack in the master inequality.  Divide
\eqref{eq:lambda2} by $x_{u^*}$, use
$\lambda=\sum_{a\in A}y_a$, and subtract the edge decomposition
\eqref{eq:mdecomp}.  This gives the exact identity
\begin{align*}
&\lambda^2-2r\lambda-\bigl(m-r(2r+1)\bigr)\\
&\quad=
\sum_{a\in A}(d_A(a)-2r)y_a-e(A)-e(B)
-\sum_{w\in B}d_A(w)(1-y_w)+r(2r+1).
\end{align*}
Splitting the sum over $A=C\cup L$ and using
\eqref{eq:degreecore} and \eqref{eq:binomialidentity}, the right-hand
side becomes
\[
\binom q2-\mathcal D.
\]
Thus the defect calculation retains the full equality information:
\[
\lambda^2-2r\lambda-\bigl(m-r(2r+1)\bigr)
=\binom q2-\mathcal D.
\]
Since $q=0$ and $\mathcal D=0$, we obtain
\[
\lambda^2-2r\lambda=m-r(2r+1).
\]
Equivalently,
\[
(\lambda-r)^2=m-r(r+1).
\]
The hypothesis \eqref{eq:star} implies $\lambda>r$, so the relevant
root is
\[
\lambda=r+\sqrt{m-r(r+1)},
\]
as required.
\end{proof}

Finally we prove that the set $B$ is empty.

\begin{lemma}\label{lem:Bempty}
	Let $G$, $u^*$, $A$, $B$, $C$, and $L$ be as fixed in Section~3. For
	all sufficiently large $m$, if $|C|=2r$ and the zero-defect
	conclusions of Lemma~\ref{lem:density}\textup{(iii)} hold, then
	$B=\varnothing$.
\end{lemma}

\begin{proof}
The purpose of this lemma is to eliminate every remaining vertex outside
the closed neighborhood of $u^*$, which completes the desired join
structure.
Suppose, to the contrary, that $B\ne\varnothing$, and fix
$w\in B$.  The zero-defect conclusions give $e(B)=0$.  Hence $w$ has
no neighbor in $B$, and by the definition of $B$ it is not adjacent to
$u^*$.  Therefore
\[
N_G(w)=N_A(w)\subseteq A.
\]
The graph $G$ has no isolated vertices, so $d_A(w)>0$.  The final
Perron-coordinate conclusion of Lemma~\ref{lem:density}\textup{(iii)}
therefore applies and yields
\begin{equation}
y_w=1,\qquad\text{or equivalently}\qquad x_w=x_{u^*}.             \label{eq:Bsame-coordinate}
\end{equation}
We now distinguish whether $w$ is adjacent to all of $A$.

First suppose that $N_A(w)=A$.  We verify all the hypotheses of the
Hub--Clique Lemma explicitly.  By the zero-defect structure,
$G[C]\cong K_{2r}$, the set $L$ is complete to $C$, and
$A=C\cup L$.  The vertex $u^*$ is complete to $A$ by the definition
$A=N_G(u^*)$, while the present assumption says that $w$ is also
complete to $A$.  Thus Lemma~\ref{lem:hub} applies with
\[
S=C,\qquad A'=L,\qquad h_1=u^*,\qquad h_2=w,\qquad
\ell=|C|=2r,\qquad j=2.
\]
It remains only to check that $L$ is large enough.  Since $y_v=1$ for
all $v\in C$ and $0<y_v\le1$ for $v\in L$, the eigenequation at
$u^*$ gives
\[
\lambda=\sum_{a\in A}y_a
=2r+\sum_{v\in L}y_v
\le2r+|L|.
\]
On the other hand, \eqref{eq:star} implies $\lambda\to\infty$ as
$m\to\infty$.  Hence $|L|\to\infty$, and for sufficiently large $m$
we have
\[
|L|\ge 2+\left\lfloor\frac{2r-2}{2}\right\rfloor=r+1.
\]
The Hub--Clique Lemma now produces
\[
2+\left\lfloor\frac{2r-2}{2}\right\rfloor
=r+1=k
\]
pairwise vertex-disjoint triangles.  This contradicts the assumption
that $G$ is $kK_3$-free.  Consequently $N_A(w)\ne A$.

We are left with $N_A(w)\subsetneq A$.  Choose a vertex
$a_0\in A\setminus N_A(w)$.  Because the Perron vector is positive,
$x_{a_0}>0$.  Using $N_G(w)=N_A(w)$ and then the eigenequations at
$w$ and $u^*$, we obtain
\begin{align*}
\lambda x_w
&=\sum_{v\in N_A(w)}x_v\\
&<\sum_{v\in A}x_v
=\lambda x_{u^*}.
\end{align*}
The inequality is strict because the second sum contains the positive
term $x_{a_0}$ that is absent from the first.  This contradicts
\eqref{eq:Bsame-coordinate}, which says that $x_w=x_{u^*}$.

Thus neither possible neighborhood relation $N_A(w)=A$ nor
$N_A(w)\subsetneq A$ can occur.  No vertex $w\in B$ exists, and hence
$B=\varnothing$.
\end{proof}

\section{Proof of the main theorem}

\begin{proof}[Proof of Theorem~\ref{thm:main}]
We argue by induction on $k$. The base case $k=2$, including its
equality characterization, is the theorem of Wang, Jia and
Ni~\cite{Wang2026}. Now fix $k\ge3$ and assume that
Theorem~\ref{thm:main} holds at level $k-1$. When $k=3$, the required
level-$(k-1)$ result is exactly the published base case; when $k\ge4$,
it is the inductive hypothesis. The definition of $M_0(k)$ at the
beginning of Section~3 ensures that every use of the level-$(k-1)$
result and every auxiliary sufficiently-large condition is valid.

Fix $m\ge M_0(k)$ and define the target value
\[
\Lambda_m:=(k-1)+\sqrt{m-k(k-1)}.
\]
Among all $kK_3$-free graphs with $m$ edges, choose a graph $G$ with
maximum spectral radius. Such a maximizer exists: after isolated
vertices are removed, an $m$-edge graph has at most $2m$ vertices, so
there are only finitely many relevant graphs. Remove all isolated
vertices from $G$. This changes neither $m$ nor $\lambda(G)$ and
preserves $kK_3$-freeness. Since every other admissible graph has
spectral radius at most $\lambda(G)$, it is enough to prove the theorem
for this extremal graph.

We first prove the upper bound. If
\[
\lambda(G)<\Lambda_m,
\]
then the desired inequality already holds strictly. We may therefore
assume
\[
\lambda(G)\ge\Lambda_m,
\]
which is precisely hypothesis \eqref{eq:star}. All the reductions from
Section~3 are now available.

Recall that $r=k-1$. Lemma~\ref{lem:finitecore} gives
\[
q=0\qquad\text{and}\qquad |C|=2r.
\]
Applying Lemma~\ref{lem:equality}, we obtain
\begin{align*}
\lambda(G)
&=r+\sqrt{m-r(r+1)}\\
&=(k-1)+\sqrt{m-k(k-1)}
=\Lambda_m.
\end{align*}
Thus the assumed lower bound can never be strict. Moreover, the
zero-defect conclusions in Lemma~\ref{lem:density}\textup{(iii)} give
\[
G[C]\cong K_{2r},\qquad e(L)=e(B)=0,
\]
and every vertex of $L$ has exactly $2r$ neighbors in $A$. Since
$|C|=2r$ and $L$ is independent, those neighbors are precisely the
vertices of $C$. Hence $L$ is complete to $C$. Finally,
Lemma~\ref{lem:Bempty} shows that
\[
B=\varnothing.
\]

We now reconstruct the whole graph. By definition,
$A=N_G(u^*)=C\cup L$, so $u^*$ is adjacent to every vertex of
$C\cup L$. The set $\{u^*\}\cup C$ therefore induces a clique of
size $2r+1$, while $L$ is independent and complete to that clique.
There are no further vertices because $B=\varnothing$ and the isolated
vertices were removed. Consequently,
\[
G\cong K_{2r+1}\vee |L|K_1
=K_{2k-1}\vee |L|K_1.
\]

Counting the clique edges and the edges between the two parts gives
\begin{align*}
m
&=\binom{2r+1}{2}+(2r+1)|L|\\
&=(2r+1)(r+|L|)\\
&=(2k-1)(k-1+|L|).
\end{align*}
It follows that
\[
(2k-1)\mid m
\qquad\text{and}\qquad
|L|=\frac{m}{2k-1}-(k-1).
\]

This also explains the nondivisible case. If $(2k-1)\nmid m$, the
assumption $\lambda(G)\ge\Lambda_m$ would force the divisibility just
derived, which is impossible. Therefore
\[
\lambda(G)<\Lambda_m
\qquad\text{whenever }(2k-1)\nmid m.
\]
If $(2k-1)\mid m$, the argument shows that any extremal graph satisfying
\eqref{eq:star} must have the displayed join structure and must attain
$\Lambda_m$.

It remains to verify that equality is attainable in the divisible
case. Suppose $(2k-1)\mid m$ and put
\[
\ell:=\frac{m}{2k-1}-(k-1).
\]
For sufficiently large $m$, we have $\ell\ge0$. By
Lemma~\ref{lem:target}, the graph
\[
K_{2k-1}\vee \ell K_1
\]
is $kK_3$-free, has exactly $m$ edges, and has spectral radius
$\Lambda_m$. Hence the upper bound is sharp whenever
$(2k-1)\mid m$.

Finally, suppose an arbitrary $m$-edge $kK_3$-free graph attains
$\Lambda_m$. It is then itself an extremal graph, so the preceding
argument applies after its isolated vertices are removed. Its unique
nontrivial component is therefore
\[
K_{2k-1}\vee
\left(\frac{m}{2k-1}-(k-1)\right)K_1.
\]
Reinstating any number of isolated vertices gives exactly the equality
family stated in Theorem~\ref{thm:main}. This completes the induction
and the proof.
\end{proof}

\section{Conclusion}

We have established, for every fixed $k\ge2$ and all sufficiently large
sizes, a sharp universal spectral upper bound for graphs without $k$ disjoint
triangles and characterized every equality case. The equality graph is unique
up to isolated vertices and is a
join of a $(2k-1)$-clique with an independent set. In particular, its
structure differs fundamentally from the balanced bipartite-type
extremizer in the corresponding fixed-order problem.

The principal methodological point is that the fixed-size problem is
governed by a second-order obstruction. The family
$K_{k-1}\vee K_{c,c}$ lies only $\Theta(m^{-1/2})$ below the target, so
an $o(m)$-edge stability statement cannot determine the extremal graph.
The proof instead passes from an exact local defect identity to a bounded
outer layer and then to a finite-core problem with independent twin
classes. The matrix identity \eqref{eq:twin-identity} turns this finite
description into Perron-vector concentration, while matching theory
identifies exactly the core set on which the normalized Perron mass
concentrates.

The theorem leaves the following two problems open.

\begin{problem}[Nondivisible sizes]\label{prob:nondivisible}
Fix $k\ge2$ and a residue
$a\in\{1,\ldots,2k-2\}$. For all sufficiently large integers
$m\equiv a\pmod{2k-1}$, determine the exact value of
\[
\max\bigl\{\lambda(G):e(G)=m,\;G\text{ is }kK_3\text{-free}\bigr\},
\]
and characterize all graphs attaining this maximum.
\end{problem}

Theorem~\ref{thm:main} gives a strict upper bound in these residue
classes, but it does not identify the optimal graph or the precise
gap from that bound.

\begin{problem}[Higher disjoint cliques]\label{prob:higher-cliques}
Fix integers $k\ge2$ and $s\ge4$. For all sufficiently large $m$,
determine the exact value of
\[
\max\bigl\{\lambda(G):e(G)=m,\;G\text{ is }kK_s\text{-free}\bigr\},
\]
and characterize all extremal graphs.
\end{problem}

An approach to Problem~\ref{prob:higher-cliques} may require an analogue
of the defect--outer-layer--finite-core reduction developed here,
together with new local packing tools, since the matching argument used
for triangles does not extend directly to $K_s$ when $s\ge4$.

\section*{Declaration on the use of AI}
The authors used generative AI tools to assist in discussing proof strategies, checking proofs, and improving the exposition. The authors take full responsibility for the mathematical arguments, results, and conclusions, all of which were carefully reviewed and verified by them.

\section*{Declaration of competing interest}

The authors declare that they have no known competing financial interests or personal relationships that could have appeared to influence the work reported in this article.

\section*{Acknowledgments}
The authors acknowledge the fruitful discussions with Hitesh Kumar, which have helped  to prove the main theorem of the article.

\section*{Data availability}

Data sharing is not applicable to this article as no datasets were generated or analyzed during the current study.

\end{document}